\documentclass[12pt,oneside,reqno]{amsart}
\usepackage[latin1]{inputenc}
\usepackage{epsfig}
\usepackage{color}
\usepackage{amsthm}
\usepackage{amsmath}
\usepackage{amsfonts}
\usepackage{amssymb}
\usepackage{cancel}
\usepackage{nicefrac}
\usepackage{hyperref}
\usepackage{graphicx}
\usepackage{epsfig}
\usepackage{floatrow}
\usepackage{lmodern}
\usepackage{ae}

\newtheorem{corollary}{Corollary}[section]

\newtheorem{lemma}[corollary]{Lemma}

\newtheorem{theorem}[corollary]{Theorem}

\newfont{\sBlackboard}{msbm10 scaled 900}

\newcommand{\mylabel}[1]{\label{#1}
            \ifx\undefined\stillediting
            \else \fbox{$#1$}\fi }
\newcommand{\BE}{\begin{equation}}

\newcommand{\EEQ}{\end{equation}}
\newcommand{\rfb}[1]{\mbox{\rm
   \eqref{#1}}\ifx\undefined\stillediting\else:\fbox{$#1$}\fi}

\newfont{\Blackboard}{msbm10 scaled 1200}

\newfont{\roma}{cmr10 scaled 1200}

\def\Frac{\displaystyle\frac}
\def\Int{\displaystyle\int}
\def\n{|\kern -.05cm{|}\kern -.05cm{|}}

\newcommand{\mm}    {{\hbox{\hskip 0.5pt}}}

\newcommand{\bluff} {{\hbox{\raise 15pt \hbox{\mm}}}}
scaled\magstep2

\makeatletter
\def\section{\@startsection {section}{1}{\z@}{-3.5ex plus -1ex minus
    -.2ex}{2.3ex plus .2ex}{\large\bf}}
\makeatother
\def\be{\begin{equation}}
\def\ee{\end{equation}}

\begin{document}
\thispagestyle{empty}
\title[viscoelastic wave equation]{Stability result  for
viscoelastic wave equation with dynamic boundary conditions}

\author{Akram Ben Aissa}
\address{UR Analysis  and Control  of PDE's, UR 13ES64,
Department of Mathematics, Faculty of Sciences of Monastir, University of Monastir, 5019 Monastir, Tunisia}
\email{akram.benaissa@fsm.rnu.tn}
\author{Mohamed Ferhat}
\address{Departement of Mathematics, Usto  University, Oran 31000,  Algeria}
\email{ferhat22@hotmail.fr}
\date{}
\begin{abstract}In this paper we consider wave viscoelastic equation  with dynamic boundary condition in a bounded domain, we establish a general decay result of energy by exploiting  the frequency domain method which consists in combining a contradiction argument and a special analysis for the resolvent of the operator of interest with assumptions on past history relaxation function. 
\end{abstract}

\subjclass[2010]{35L05;35L15; 35L70; 93D15}
\keywords{ Energy decay; infinite memory; dynamic boundary condition..}

\maketitle
 
\tableofcontents

\vfill\break
 
\section{Introduction}
We omit the space variable $x$ of $u(x,t)$,\ $u_{t}(x,t)$  and for simplicity reason denote $u(x,t)=u$, $u_{t}(x,t)=u_{t}$, when no confusion arises also the functions considered are all real valued, here $u_{t}=\partial u(t)/\partial t$,
$u_{tt}=\partial^{2}u(t)/\partial t^{2}$.\\
Our main interest lies in the following system of viscoelastic
equation :
\begin{equation}\label{pb1}
\left\{
\begin{aligned}
&u_{tt}-\Delta u+\Int_{0}^{\infty}g(s)\Delta u(x,t-s)ds=0, \ &x\in\Omega&, \ t>0 \\
&u_{tt}=-\left(\Frac{\partial u}{\partial
\nu}(x,t)-\Int_{0}^{\infty}g(s)\Frac{\partial u}{\partial
\nu}(x,t-s)ds\right), \  &x\in\Gamma_{1}&, \ t>0\\
&u(x,t)=0, \  &x\in\Gamma_{0}&, \ t>0 \\
&u(x,-t)=u_{0}(x,t),  \ &x\in\Omega&, \ t>0 \\
&u_{t}(x,0)=u_{1}(x), \ &x\in\Omega&, \\
&u(x,0)=u_{0}(x), \  &x\in\Omega&,
\end{aligned}
\right.
\end{equation}
The main difficulty of the problem considered is related to the non ordinary boundary conditions defined on $\Gamma_{1}$. Very
little attention has been paid to this type of boundary conditions. From the mathematical point of view, these problems do
not neglect acceleration terms on the boundary. Such types of boundary conditions are usually called dynamic boundary
conditions. They are not only important from the theoretical point of view but also arise in several physical applications.
For instance in one space dimension, problem (\ref{pb1}) can modelize the dynamic evolution of a viscoelastic rod that is fixed at
one end and has a tip mass attached to its free end. The dynamic boundary conditions represent the Newton's law for the
attached mass, (see \cite{F10,F12} for more details) which arise when we consider the transverse motion of a flexible membrane
whose boundary may be affected by the vibrations only in a region. Also some of them as in problem (\ref{pb1}) appear when we
assume that is an exterior domain of $\mathbb{R}^{3}$ in which homogeneous fluid is at rest except for sound waves. Each point of the
boundary is subjected to small normal displacements into the obstacle (see \cite{F12} for more details). Among the early results
dealing with the dynamic boundary conditions are those of Grobbelaar-Van Dalsen \cite{F2,F3} in which the authors have made
contributions to this field and in \cite{F4} the authors have studied the following problem:

\begin{equation}
\left\{
\begin{aligned}
&u_{tt}-\Delta u+\delta \Delta u_{t}=|u|^{p-1}u, \ &x\in\Omega&, \ t>0 \\
&u_{tt}=-\left(\Frac{\partial u(x,t)}{\partial
\nu}(x,t)+\delta \Frac{\partial u(x,t)}{\partial
\nu}(x,t)+\alpha|u_{t}|^{m-1}u(x,t)\right), \  &x\in\Gamma_{1}&, \ t>0\\
&u(x,t)=0, \  &x\in\Gamma_{0}&, \ t>0 \\
&u_{t}(x,0)=u_{1}(x), \ &x\in\Omega&, \\
&u(x,0)=u_{0}(x), \  &x\in\Omega&,
\end{aligned}
\right.
\end{equation}
and they have obtained several results concerning local existence which extended to the global existence by using the
concept of stable sets, the authors have obtained also the energy decay and the blow up of the solutions for positive initial
energy.\\
The same problem has traited by \cite{F7}, they showed the existence and uniqueness of a local in time solution and under some restrictions on the initial data, the solution continues to exist globally in time. On the other hand, if the interior source dominates the boundary damping, they proved that the solution is unbounded and grows as an exponential function. In addition, in the absence of the strong damping, they proved also the solution ceases to exist and blows up in finite time. Related problem as \cite{F5}, M. M. Cavalcanti, A. Khemmoudj and M. Medjden \cite{F18} studied the following system:
 
\begin{eqnarray}
\left\{
\begin{aligned}
&u_{tt}+\mathcal{A}u+a(x)g_{1}(u_{t})=0, \ &x&\in\Omega, \ t>0,\nonumber \\
&u_{tt}+\Frac{\partial u(x,t)}{\partial
\nu_{\mathcal{A}}}+{\mathcal{A}}_{T}\upsilon+g_{2}(\upsilon_{t})=0, \  &x&\in\Gamma_{1}, \ t>0,\nonumber\\
&u(x,t)=0 \  &x&\in\Gamma_{0}, \ t>0, \nonumber\\
&u(x,t)=\upsilon, \  &x&\in\Gamma_{1}, \ t>0, \\
&(u_{t}(x,0),\upsilon_{t}(x,0))=(u_{1},\upsilon^{1}), \ &x&\in(\Omega,\Gamma_{1}),  \nonumber \\
&(u(x,0),\upsilon(0))=(u_{0}(x),\upsilon^{1}),   \   &x&\in(\Omega,\Gamma_{1}).
\end{aligned}
\right.
\end{eqnarray}
They supposed that the second-order differential operators $\mathcal{A}$
and $\mathcal{A}_{T}$ satisfy certain uniform ellipticity conditions, and they obtained uniform stabilization by using Riemannian geometry methods.\\
Motivated by the previous works, it is interesting to show more general decay result to that in \cite{F5} and \cite{F6}, we analyze the
influence of the viscoelastic,  on the solutions to (\ref{pb1}). Under suitable assumption on  function
$g(.)$, the initial data and the parameters in the equations.\\
 The content of this paper is organized as follows: In
Section 2, we provide assumptions that will be used later. In Section 3, we state and prove the local existence result. In Section 4,  by exploiting the frequency domain method we prove the stability result.

\section{Preliminaries}
In this section, we present some material and assumptions for the
proof of our
results.\\
Denote
$$
H_{\Gamma_{0}}^{1}(\Omega)=\left\{u\in H^{1}(\Omega):
u_{\Gamma_{0}}=0\right\},
$$

$$
H_{\Gamma_{0}}^{1}(\Gamma)=\left\{u\in H^{1}(\Gamma):
u_{\Gamma_{0}}=0\right\},
$$

we set $\gamma_{1}$ the trace operator from
$H_{\Gamma_{0}}^{1}(\Omega)$ on $L^{2}(\Gamma_{1})$ and
$H^{\frac{1}{2}}(\Gamma_{1})=\gamma_{1}(H_{\Gamma_{0}}^{1}(\Omega)).$
We denote by $B$ the norm of $\gamma_{1}$ namely:
$$
\forall u\in H_{\Gamma_{0}}^{1}(\Omega), \quad
\|u\|_{2.\Gamma_{1}}\leq B\|\nabla u\|_{2}.
$$
We will use the following embeddings
\begin{eqnarray*}
H_{\Gamma_{0}}^{1}(\Omega)&\hookrightarrow& L^{q}(\Omega) \;\text{for}\;\; 2\leq
q\leq \frac{2n}{n-2}, \;\;\text{if}\; n\geq 3 \;\text{and}\; q\geq 2, \;\text{if}\; n=1,2\\
L^{r}(\Omega)&\hookrightarrow &L^{q}(\Omega),\quad\text{for}\; q<r.
\end{eqnarray*}
 Then for some  $ c_{s}>0$,
$$
\|\nu\|_{q}\leq c_{s}\|\nabla \nu\|_{2},  \quad \|\nu\|_{q}\leq
c_{s}\| \nu\|_{r}\quad \hbox{for}\quad \nu\in
H_{\Gamma_{0}}^{1}(\Omega).
$$
We recall that $H^{\frac{1}{2}}(\Gamma_{1})$ is dense in
$L^{2}(\Gamma_{1})$. We denote
$$
E(\Delta, L^{2}(\Omega))=\left\{u\in H^{1}(\Omega) \ \ \hbox{such that}  \ \Delta u\in L^{2}(\Omega) \right\},
$$
recall that for a function $u\in E(\Delta, L^{2}(\Omega))$, \ $\Frac{\partial u}{\partial\nu}\in H^{-\frac{1}{2}}(\Gamma_{1})$.
We will usually use the following Green's formula
\begin{eqnarray}\label{2-5} 
\!\!\!\int_\Omega\nabla
u(x)\nabla\omega(x)dx=-\int_\Omega\Delta
u(x)\omega(x)dx+\int_{\Gamma_1}\frac{\partial
u}{\partial\nu}(x) \omega(x)d\Gamma_{1},\;\forall\omega\in
H^1_{\Gamma_0}(\Omega).
\end{eqnarray}
 For studying the problem (\ref{pb1}) we will need the
following assumptions (A1).
\begin{itemize}
\item  The relaxation function  $g$ is differentiable function such that, for $s\geq0$

\begin{equation}
g(s)\geq0,\quad 1-\int_{0}^{\infty}g(s)ds=\ell>0,
\end{equation}

\item 
\begin{equation}
\exists\; \zeta_{0},  \zeta_{1}>0: \quad  -\zeta_{1}g(t)\leq g'(t)\leq -\zeta_{0}g(t),\quad \forall \ t\in \mathbb{R}.
\end{equation}
\end{itemize}
\section{Well-posedness of the problem }
In order to prove the existence  of solutions of problem
(\ref{pb1}), we follow the approach of Dafermos  \cite{Daf}, by considering a new
auxiliary variable the relative history of $u$  as follows:
$$
\eta:=\eta^{t}(x,s)=u(x,t)-u(x,t-s)\quad  \hbox{in} \quad
\Omega\times (0,\infty)\times (0,\infty).
$$
and the weighted $L^{2}-$ spaces
$$
\begin{aligned}
\mathcal{M}&=L_{g}^{2}(\mathbb{R_{+}}; H_{\Gamma_{0}}^{1}(\Omega))\\&=\left\{\xi:\mathbb{R}_{+} \to
 H_{\Gamma_{0}}^{1}(\Omega)): \Int_{0}^{\infty}g(s)\|\nabla
\xi(s)\|_{2}^{2}ds<\infty\right\},
\end{aligned}
$$
which is a Hilbert space endowed with inner product  and norm consecutively
$$
\left\langle
\xi,\zeta\right\rangle_{\mathcal{M}}=\Int_{0}^{\infty}g(s)\left(\Int_{\Omega}\nabla
\xi(s)\nabla\zeta(s)dx\right)ds,
$$
and
$$
\|\xi\|_{\mathcal{M}}^{2}=\Int_{0}^{\infty}g(s)\|\nabla
\xi(s)\|_{2}^{2}ds.
$$
Our analysis is given on the phase space

\begin{eqnarray}
\mathcal{H}= H_{\Gamma_{0}}^{1}(\Omega)\times
\times L^{2}(\Omega)\times
L^{2}(\Gamma_{1})\times\mathcal{M}.
\end{eqnarray}
If we denote $V:=(u,u_{t},\gamma_{1}(u_{t}),\eta)$, clearly, $\mathcal{H}$ is a Hilbert space with  respect  to the inner product
\begin{eqnarray}
\displaystyle\langle V_{1},V_{2}\rangle_{\mathcal{H}}&=&(1-g_{0})\Int_{\Omega}\nabla u_{1}.\nabla u_{2}dx+\Int_{\Omega} \upsilon_{1}. \upsilon_{2}dx+\Int_{\Gamma_{1}} w_{1}. w_{2}d\sigma\nonumber\\
&+&\Int_{0}^{\infty}g(s)\left(\Int_{\Omega}\nabla
\eta_{1}(s).\nabla\eta_{2}(s)dx\right)ds,
\end{eqnarray}
for $V_{1}=(u_{1},\upsilon_{1},w_{1},\eta_{1})^\mathsf{T}$ and $V_{2}=(u_{2},\upsilon_{2},w_{2},\eta_{2})^\mathsf{T}.$
Therefore, problem  (\ref{pb1}) is equivalent to
\begin{equation}\label{pb2}
\left\{
\begin{aligned}
&u_{tt}-\ell\Delta u-\Int_{0}^{\infty}g(s)\Delta
\eta^{t}(x,s)ds=0, \ &x\in\Omega&, \ t>0, \\
&u_{tt}=-\left(\Frac{\partial u}{\partial
\nu}(x,t)+\Int_{0}^{\infty}g(s)\Frac{\partial u}{\partial
\nu}(x,t-s)ds\right), \ &x\in\Gamma_{1}&, \ t>0,\\
&\eta_{t}^{t}(x,t)+\eta_{s}^{t}(x,s)=u_{t}(x,t), \
&x\in\Omega&, \ t>0, \  s>0,\\
&u(x,t)=\eta^{t}(x,0)=0, &x\in\Gamma_{0}&, \ t>0,\\
&u(x,-t)=u_{0}(x,t), &x\in\Omega&, \ t>0,\\
&u_{t}(x,0)=u_{1}(x), \ &x\in\Omega&,\\
&u(x,0)=u_{0}(x),\ &x\in\Omega&.
\end{aligned}
\right.
\end{equation}
If $V_{0}\in\mathcal{H}$ and  $V\in\mathcal{H}$ , the problem (\ref{pb2}) is formally equivalent to the following abstract evolution equation in the Hilbert space  $\mathcal{H}$

\begin{equation}\label{cp11}
\left\{
\begin{array}{ll}
V'(t)=\mathcal{A}V(t), \quad t>0\\
\\
V(0=V_{0},.
\end{array}
\right.
\end{equation}
such that  $V_{0}=(u_{0},u_{1},\gamma_{1}(u_{1}),\eta^{0})^\mathsf{T}$ and the operator $\mathcal{A}$ is defined by

\begin{equation}
\mathcal{A}\left(
\begin{array}{c}
u \\
\upsilon \\
\omega \\
\eta \\
\end{array}
\right)
=\left(
   \begin{array}{c}
     \upsilon \\
     (1-g_{0})\Delta u+\Int_{0}^{\infty}g(s)\Delta \eta(s)ds \\
     -\Frac{\partial u}{\partial\nu}-\Int_{0}^{\infty}g(s)\Frac{\partial \omega}{\partial\nu}(x,t-s)ds \\
     -\Frac{\partial \eta}{\partial s}+\upsilon \\
   \end{array}
 \right)
\end{equation}
The domain of $\mathcal{A}$ is the set of $V=(u,\upsilon,\omega,\eta)^\mathsf{T}$ such that the domain of $\mathcal{A}$ is defined by
\begin{equation*}
\displaystyle
D(\mathcal{A})=\left\{
\begin{array}{ll}
(u,\upsilon,\omega,\eta)\in \left(H^{2}(\Omega)\cap H_{\Gamma_{0}}^{1}(\Omega)\right)\times H_{\Gamma_{0}}^{1}(\Omega)\times L^{2}(\Gamma_{1})\times\mathcal{M},\\
(1-g_{0})u+\Int_{0}^{\infty}g(s)\eta(s)ds\in L^{2}(\Omega), \  \omega=\gamma_{1}(u)= u_{0}(.,0), \ \eta(0)=0 \ \hbox{on} \ \Gamma_{1}
  \end{array}
  \right\}
\end{equation*}

Now, our main result is stated as follows:
\begin{theorem}
Let $V_{0}\in \mathcal{H}$. Then,  system (\ref{pb2}) has a unique weak solution
$$
V\in \mathcal{C}(\mathbb{R}^{+}; \mathcal{H})
$$
Moreover, if $V_{0}\in D(\mathcal{A})$, then the solution of (\ref{cp11}) satisfies
$$
V\in \mathcal{C}^{1}(\mathbb{R}^{+}; \mathcal{H})\cap \mathcal{C}(\mathbb{R}^{+}; \mathcal{H})
$$
\end{theorem}

\begin{proof}
By Lumer-Phillips'Theorem, it suffices to show that $\mathcal{A}$  is m-dissipative.\\
We first prove that $\mathcal{A}$  is dissipative. Indeed,  for any $V=(u,\upsilon,\omega,\eta)^\mathsf{T}\in D(\mathcal{A})$, we have 

\begin{equation}\label{eq4}
\begin{aligned}
\langle \mathcal{A}V,V\rangle_{\mathcal{H}}&=\left\langle\left(
   \begin{array}{c}
     \upsilon \\
     (1-g_{0})\Delta u+\Int_{0}^{\infty}g(s)\Delta \eta(s)ds \\
     -\Frac{\partial u}{\partial\nu}-\Int_{0}^{\infty}g(s)\Frac{\partial u}{\partial\nu}(x,t-s)ds \\
     -\Frac{\partial \eta}{\partial s}+\upsilon \\
   \end{array}
 \right),\left(
\begin{array}{c}
u \\
\upsilon \\
\omega \\
\eta \\
\end{array}
\right)\right\rangle
\\&=(1-g_{0})\Int_{\Omega}\nabla \upsilon.\nabla u dx+(1-g_{0})\left\{\int_{\Omega}\Delta u.\upsilon dx+\int_{\Omega}\int_{0}^{\infty}g(s)\Delta \eta(s)\upsilon(s)ds\right\}
\\&+\Int_{\Gamma_{1}}\left(\frac{-\partial u}{\partial\nu}-\int_{0}^{\infty}g(s)\frac{\partial u}{\partial\nu}(x,t-s)ds\right)\omega d\sigma+\left\langle \frac{-\partial \partial\eta}{\partial s}+\upsilon,\eta\right\rangle_{L_{g}^{2}}
\end{aligned}
\end{equation}
Noting that
\begin{equation}
\Int_{\Gamma_{1}}\left(\frac{-\partial u}{\partial\nu}-\int_{0}^{\infty}g(s)\frac{\partial u}{\partial\nu}(x,t-s)ds\right)\omega d\sigma=0
\end{equation}
By exploiting Green's formula and integrating by parts and using the fact that $\eta(0)=0$ (definition of $D(\mathcal{A})$), we obtain

$$
\left\langle \frac{-\partial \eta}{\partial s},\eta\right\rangle_{L_{g}^{2}}=\frac{1}{2}\int_{0}^{\infty}g'(s)\|\nabla\eta(s)\|^{2}.ds
$$
Inserting the previous inequalities into (\ref{eq4}) , we obtain
$$
\langle \mathcal{A}V,V\rangle_{\mathcal{H}}=\frac{1}{2}\int_{0}^{\infty}g'(s)\|\nabla\eta(s)\|^{2}ds,
$$
which implies that
$$
\langle \mathcal{A}V,V\rangle_{\mathcal{H}}\leq0,
$$
since $g$ is nonincreasing. This means that $\mathcal{A}$ is dissipative. Note that, thanks to (A1) and the fact that $\eta\in L_{g}^{2}(\mathbb{R}; H_{\Gamma_{0}}^{1}(\Omega))$,

\begin{equation}
\begin{aligned}
\left|\int_{0}^{\infty}g'(s)\|\nabla\eta(s)\|^{2}ds\right|&=-\int_{0}^{\infty}g'(s)\|\nabla\eta(s)\|^{2}ds\\&
\leq \zeta_{1}\int_{0}^{\infty}g(s)\|\nabla\eta(s)\|^{2}ds\\&
< +\infty,
\end{aligned}
\end{equation}
Next, we shall prove that $I\lambda-\mathcal{A}$ is surjective for $\lambda>0$. Indeed, let $F=(f_{1},f_{2},f_{3},f_{4})^\mathsf{T}\in \mathcal{H}$, we show that there exists $W=(\omega_{1}, \omega_{2}, \omega_{3},\omega_{4})^\mathsf{T}\in D(\mathcal{A})$ satisfying
\begin{equation}\label{res1}
(I\lambda-\mathcal{A})W=F
\end{equation}
As previously , we have
$$
\mathcal{A}=
\left(
  \begin{array}{cccc}
    0 & I & 0 & 0 \\
    (1-g_{0})\Delta & 0 & 0 & \Int_{0}^{\infty}g(s)\Delta ds \\
    \Frac{-\partial }{\partial \nu} & 0 & -\Int_{0}^{\infty}g(s)\Frac{\partial }{\partial \nu} ds & 0 \\
    0 & I & 0 & -\Frac{\partial}{\partial s} \\
  \end{array}
\right)
$$
which gives us 
\begin{equation}\label{ser1}
\left\{
\begin{array}{ll}
\lambda \omega_{1}-\omega_{2}=f_{1}\\
-(1-g_{0})\Delta \omega_{1}+\lambda\omega_{2}-\Int_{0}^{\infty}g(s)\Delta \omega_{4}(s)ds=f_{2}\\
 \lambda\omega_{3}+\Frac{\partial \omega_{1}}{\partial \nu}+\Int_{0}^{\infty}g(s)\Frac{\partial \omega_{3}(s)}{\partial \nu}ds=f_{3}\\
-\omega_{2}+\lambda \omega_{4}+\Frac{\partial}{\partial s }\omega_{4}=f_{4}
\end{array}
\right.
\end{equation}
we note that the first in (\ref{ser1}) equation gives

\begin{equation}\label{qa2}
\omega_{2}=\lambda \omega_{1}-f_{1}
\end{equation}
and the last equation in (\ref{ser1}) with $\eta(0)=0$ has unique solution
\begin{equation}\label{qa1}
\omega_{4}(s)=\left(\Int_{0}^{s}e^{y}(f_{4}(y)+\omega_{2}(y))dy\right)e^{-s}.
\end{equation}
From the first and the second equation in (\ref{ser1}) we can deduce the following

\begin{equation}\label{qi1}
\lambda^{2}\omega_{1}-(1-g_{0})\Delta\omega_{1}=(f_{2}+\lambda f_{1})+\Int_{0}^{\infty}g(s)\omega_{4}(s)ds.
\end{equation}
Putting $\bar{u}=\omega_{1}+\Int_{0}^{\infty}g(s)\omega_{3}(s)ds$. Then  from equation (\ref{qi1}),  $\bar{u}$ must satisfy

\begin{equation}\label{qi2}
\begin{aligned}
\lambda^{2}\bar{u}-(1-g_{0})\Delta \bar{u}&=\lambda^{2}\int_{0}^{\infty}g(s)\omega_{3}(s)ds-(1-g_{0})\int_{0}^{\infty}g(s)\Delta\omega_{3}(s)ds\\&+(f_{2}+\lambda f_{1})+\Int_{0}^{\infty}g(s)\omega_{4}(s)ds
\end{aligned}
\end{equation}
with the boundary conditions
\begin{eqnarray}
\bar{u}=0 \quad &\hbox{on}&\quad \Gamma_{0}\\
\Frac{\partial \bar{u}}{\partial \nu}=f_{3}-\lambda \bar{u}+\lambda u_{0}(x)(1-l) \quad &\hbox{on}& \quad \Gamma_{1}.
\end{eqnarray}
It is sufficient to prove that (\ref{qi2}) has a solution $\bar{u}$  in $H^{2}\cap H_{\Gamma_{0}}^{1}(\Omega)$ and  replacing it in (\ref{qa1}) and (\ref{qa2}) to conclude that (\ref{pb2}) has a solution $V\in D(\mathcal{A})$. So we multiply (\ref{qi2}) by a test function $\varphi\in H_{\Gamma_{0}}^{1}(\Omega)$ and we integrate by parts, obtaining the following variational formulation of (\ref{qi2}):
\begin{equation}\label{fi1}
a(\bar{u},\varphi)=l(\varphi) \quad \forall \ \varphi\in H_{\Gamma_{0}}^{1}(\Omega)
\end{equation}
where
\begin{equation}\label{fi2}
a(\bar{u},\varphi)=\Int_{\Omega}\left[\lambda^{2}\bar{u}.\varphi+(1-g_{0)})\nabla \bar{u}.\nabla\varphi \right]dx+\lambda\int_{\Gamma_{1}}\bar{u}(\sigma)\varphi(\sigma)d\sigma
\end{equation}

and
\begin{equation}\label{fi3}
\begin{aligned}
l(\varphi)&=\Int_{\Omega}\left[\lambda^{2}\Int_{0}^{\infty}g(s)\omega_{3}(s)ds\varphi dx+(1-g_{0)}\int_{0}^{\infty}g(s)\nabla\omega_{3}(s)ds\nabla\varphi dx+(f_{2}+\lambda f_{1})\varphi dx\right] \\&+\Int_{\Omega}\Int_{0}^{\infty}g(s)\omega_{4}(s)ds \varphi dx+\lambda\int_{\Gamma_{1}}u_{0}(\sigma)\varphi(\sigma)d\sigma
\end{aligned}
\end{equation}
It is clear that a is a bilinear and continuous form on $H_{\Gamma_{0}}^{1}(\Omega)$ and l is linear and continuous form on $H_{\Gamma_{0}}^{1}(\Omega)$. On the other hand, (\ref{fi2}) implies that there exists a positive constant $a_{0}$ such that

\begin{equation}
\begin{aligned}
a(\bar{u},\bar{u})&=\Int_{\Omega}\lambda^{2}|\bar{u}|^{2}dx+(1-g_{0)})\int_{\Omega}|\nabla\bar{u}|^{2}dx+\lambda\int_{\Gamma_{1}}|\bar{u}(\sigma)|^{2}d\sigma\\&
\geq a_{0}\|\bar{u}\|_{2}^{2} \  \  \forall \bar{u}\in H_{\Gamma_{0}}^{1}(\Omega),
\end{aligned}
\end{equation}
which implies that $a$ is coercive. Therefore, using the Lax-Milgram Theorem, we conclude that (\ref{qi2}) has a unique solution $\bar{u}$ in $H_{\Gamma_{0}}^{1}(\Omega)$. By classical regularity arguments, we conclude that the solution $\bar{u}$ of (\ref{qi2}) belongs into $H^{2}(\Omega)\cap H_{\Gamma_{0}}^{1}(\Omega)$ and satisfies (\ref{qi2}). Consequently, using  (\ref{qa2}) and (\ref{qa1}), we deduce  that (\ref{pb2}) has a unique solution $V\in D(\mathcal{A})$. This proves that $(\lambda I-\mathcal{A})$ is surjective and hence $\mathcal{A}$ is an infinitesimal generator of a linear $\mathcal{C}_{0}$ semigroup of contractions on $\mathcal{H}$.
\end{proof}
The energy associated with (\ref{pb2}) is defined by

\begin{equation}\label{en1}
E(t)=\frac{1}{2}\left\{\|u_{t}(t)\|_{2}^{2}+\|\nabla
u(t)\|_{2}^{2}+\|\eta\|_{L_{g}^{2}}^{2}\right\},
\end{equation}

\begin{lemma} The functional defined in (\ref{en1}) satisfies the following
inequality

\begin{equation}\label{mir1}
E'(t)\leq\Frac{1}{2}\int_{0}^{\infty}g'(s)\|\nabla
\eta(s)\|_{2}^{2}ds , \  \forall \  t\geq0,
\end{equation}
\end{lemma}
\begin{proof}
By multiplying the first equation in (\ref{pb2}) by $u_{t}(t)$,
and integrating over $\Omega$ we get

\begin{equation}\label{eq1}
\frac{1}{2}\frac{d}{dt}\left\{\|u_{t}(t)\|_{2}^{2}+\frac{1}{2}\|\nabla
u(t)\|_{2}^{2}\right\}+\int_{0}^{\infty}g(s)\int_{\mathbb{R}^{n}}\nabla \eta(s)\nabla
u_{t}(t)dsdx.
\end{equation}
Since
$$u_{t}(x,t)=\eta_{t}(x,s)+\eta_{s}(x,s), \quad
(x,s)\in \Omega \times \mathbb{R}^{+}, \ t\geq0,$$
we have
\begin{equation}\label{eq2}
\begin{aligned}
\int_{0}^{\infty}g(s)\int_{\Omega}\nabla \eta(s)\nabla
u_{t}(t)dxds&=\int_{0}^{\infty}g(s)\int_{\Omega}\nabla
\eta(s)\nabla
\eta_{t}(t)dxds\\&+\int_{0}^{\infty}g(s)\int_{\Omega}\nabla
\eta(s)\nabla
\eta_{s}(t)dxds\\&=\Frac{1}{2}\int_{0}^{\infty}g(s)\Frac{d}{dt}\|\nabla
\eta(s)\|_{2}^{2}ds\\&-\Frac{1}{2}\int_{0}^{\infty}g'(s)\|\nabla
\eta(s)\|_{2}^{2}ds\\&
+\int_{0}^{\infty}g(s)\int_{\Omega}\nabla
\eta(s)\nabla
\eta_{t}(t)dxds,
\end{aligned}
\end{equation}
Due to Young's inequality, we have for any $\delta>0$

\begin{equation}\label{eq3}
\begin{aligned}
\int_{0}^{\infty}g(s)\int_{\Omega}\nabla
\eta(s)\nabla
\eta_{t}(t)dxds&\leq\int_{0}^{\infty}g(s)\left(\frac{1}{4\delta}\|\nabla\eta(s)\|_{2}^{2}+\delta\|\nabla
\eta_{t}\|_{2}^{2}\right)ds\\&
\leq\delta\left(\int_{0}^{\infty}g(s)ds\right)\|\nabla
\eta_{t}\|_{2}^{2}+\frac{1}{4\delta}\int_{0}^{\infty}g(s)\|\nabla\eta(s)\|_{2}^{2}ds
\\&
=\delta g_{0}\|\nabla
\eta_{t}\|_{2}^{2}+\frac{1}{4\delta}\|\eta\|_{L_{g}^{2}}^{2},
\end{aligned}
\end{equation}
by replacing (\ref{eq2}) and  (\ref{eq3}) into (\ref{eq1}) we get the desired result.

\end{proof}

\section{Stability result}
The necessary and sufficient conditions for the exponential stability of the $C_0$-semigroup of contractions on a Hilbert space were obtained by Gearhart \cite{Ge} and Huang \cite{Hu} independently, see also Prüss \cite{Pr}. We will use the following result due to Gearhart.

\begin{lemma}
A Semigroup $\{e^{t\mathcal{A}}\}_{t\geq0}$ of contractions on a Hilbert space $\mathcal{X}$  is exponentially stable if and only if
\begin{equation}\label{ak1}
i\mathbb{R}\equiv\{i\beta; \quad \beta\in \mathbb{R} \}\subset\rho(\mathcal{A})
\end{equation}
and
\begin{equation}\label{ak2}
\limsup_{|\beta|\to \infty} \|{(i\beta I-\mathcal{A})}^{-1}\|_{\mathcal{X}}< \infty
\end{equation}
\end{lemma}
Our main result reads as follows:
\begin{theorem}
The semigroup of system (\ref{pb2}) decays exponentially as 
\begin{equation}\label{free1}
\|e^{t\mathcal{A}}V_{0}\|_{\mathcal{H}}\leq C e^{-\gamma t}\|V_{0}\|_{D(\mathcal{A})}, \quad \forall\;V_0\in D(\mathcal{A}),\ t>0
\end{equation}
\end{theorem}
\begin{proof}
The proof is splinted  into two parts the first part consists to prove  (\ref{ak1}) which is equivalent to prove the following two assertions

\begin{enumerate}
\item If $\beta$ is a real number, then $(i\beta I-\mathcal{A})$ is injectif and
\item If $\beta$ is a real number, then $(i\beta I-\mathcal{A})$ is surjectif.
\end{enumerate}
It is the objective of the two following lemmas.
\begin{lemma}
If $\beta$ is a real number, then $i\beta$ is not an eigenvalue of $\mathcal{A}$
\end{lemma}
\begin{proof}
We will show that the equation
\begin{equation}\label{ak3}
\mathcal{A}Z=i\beta Z
\end{equation}
with $Z=(u,\upsilon,\omega,\eta)^\mathsf{T}\in D(\mathcal{A})$ and $\beta\in \mathbb{R}$ has only the trivial solution. Equation (\ref{ak3}) can be written as
\begin{eqnarray}
&i\beta u-\upsilon=0\label{ri1}\\
&i\beta \upsilon-(1-g_{0})\Delta u-\int_{0}^{\infty}g(s)\Delta \eta(s)ds=0\label{ri2}\\
&i\beta \omega+\Frac{\partial u}{\partial\nu}+\int_{0}^{\infty}g(s)\Frac{\partial \omega(s)}{\partial \nu}ds=0\label{ri3}\\
&i\beta\eta+\Frac{\partial \eta}{\partial s}-\upsilon=0\label{ri4}
\end{eqnarray}
By taking the inner product of (\ref{ak3}) with $Z\in D(\mathcal{A})$ and using (\ref{mir1}), we get:
\begin{equation}
\begin{aligned}
\Re(\left\langle \mathcal{A}Z,Z\right\rangle_{\mathcal{H}})&\leq \int_{0}^{\infty}g'(s)\|\nabla\eta(s)\|^{2}ds\\&
\leq-\int_{0}^{\infty}g(s)\|\nabla\eta(s)\|^{2}ds\\&
=-\|\eta\|_{\mathcal{M}}^{2}\\&
\leq0
\end{aligned}
\end{equation}
Thus we obtain that: $\eta=0$, moreover as $\eta$ satisfies (\ref{ri4})  by integration, we obtain

$$
\eta(s)=\left(\Int_{0}^{s}e^{i\beta y}\upsilon(y))dy\right)e^{-i\beta s}
$$
since $\eta=0$ we deduce  that $\upsilon=0$  and from (\ref{ri1}) we have $u=0$. since $\omega=\gamma_{1}(u)=u_{0}(.,0)$, we obtain also $\omega=0$. Thus the only solution of (\ref{ak3}) is the trivial one. Hence the proof is completed.
\end{proof}
Next, we show that $\mathcal{A}$ has no continuous spectrum on the imaginary axis.
\begin{lemma}
If $\beta$ is a real number, then $i\beta$ to the resolvent set $\rho(\mathcal{A})$ of $\mathcal{A}$
\end{lemma}
\begin{proof}
In view of Lemma  it is enough to show that $\mathcal{A}$ is surjective.

For $F=(f_{1},f_{2},f_{3},f_{4})^\mathsf{T}\in \mathcal{H}$, let $V=(u,\upsilon,\omega,\eta)^\mathsf{T}\in D(\mathcal{A})$ solution of
\begin{equation}\label{kari}
(i\beta I-\mathcal{A})V=F
\end{equation}
which is
\begin{equation}\label{tas1}
\left\{
\begin{array}{ll}
i\beta u-\upsilon=f_{1}\\
-(1-g_{0})\Delta u+i\beta \upsilon-\Int_{0}^{\infty}g(s)\Delta \eta(s)ds=f_{2}\\
 i\beta\omega+\Frac{\partial u}{\partial \nu}+\Int_{0}^{\infty}g(s)\Frac{\partial \omega(s)}{\partial \nu}ds=f_{3}\\
-\upsilon+i\beta \eta+\Frac{\partial \eta }{\partial s }=f_{4}
\end{array}
\right.
\end{equation}
we note that the first equation in (\ref{tas1}) gives

\begin{equation}\label{tas2}
\upsilon=i\beta \omega_{1}-f_{1}.
\end{equation}
The last equation in (\ref{tas1}) with $\eta(0)=0$ has unique solution
\begin{equation}\label{tass2}
\omega_{4}(s)=\left(\Int_{0}^{s}e^{i\beta y}(f_{4}(y)+\omega_{2}(y))dy\right)e^{-i\beta s}
\end{equation}
Another time, from the first and the second equation in (\ref{tas1}) we can deduce the following

\begin{equation}\label{tas5}
(i\beta)^{2}\omega_{1}-(1-g_{0})\Delta\omega_{1}=(f_{2}+i\beta f_{1})+\Int_{0}^{\infty}g(s)\omega_{4}(s)ds
\end{equation}
If we take $\omega_{1}+\Int_{0}^{\infty}g(s)\omega_{3}(s)ds=\bar{u}$, then  from equation (\ref{tas5})  $\bar{u}$ must satisfy

\begin{equation}\label{tas6}
\begin{aligned}
(i\beta)^{2}\bar{u}-(1-g_{0})\Delta \bar{u}&=(i\beta)^{2}\int_{0}^{\infty}g(s)\omega_{3}(s)ds-(1-g_{0})\int_{0}^{\infty}g(s)\Delta\omega_{3}(s)ds\\&+(f_{2}+i\beta f_{1})+\Int_{0}^{\infty}g(s)\omega_{4}(s)ds
\end{aligned}
\end{equation}
with the boundary conditions
\begin{eqnarray}
\bar{u}=0 \quad &\hbox{on}&\, \Gamma_{0}\\
\Frac{\partial \bar{u}}{\partial \nu}=f_{3}-i\beta \bar{u}+i\beta u_{0}(x)(1-l) \quad &\hbox{on}&\, \Gamma_{1}.
\end{eqnarray}
It is sufficient to prove that (\ref{tas6}) has a solution $\bar{u}$  in $H^{2}\cap H_{\Gamma_{0}}^{1}(\Omega)$ and then we replace in (\ref{tas2}) and (\ref{tass2}) to conclude that (\ref{kari}) has a solution $V\in D(\mathcal{A})$. Then multiplying (\ref{tas6}) by a test function $\varphi\in H_{\Gamma_{0}}^{1}(\Omega)$ and we integrate by parts, obtaining the following variational formulation of (\ref{tas6}):
\begin{equation}\label{mil1}
b(\bar{u},\varphi)=l(\varphi) \quad \forall \ \varphi\in H_{\Gamma_{0}}^{1}(\Omega)
\end{equation}
where
\begin{equation}\label{mil2}
b(\bar{u},\varphi)=\Int_{\Omega}\left[(i\beta)^{2}\bar{u}.\varphi+(1-g_{0)})\nabla \bar{u}.\nabla\varphi \right]dx+i\beta\int_{\Gamma_{1}}\bar{u}(\sigma)\varphi(\sigma)d\sigma
\end{equation}

and
\begin{equation}\label{mil3}
\begin{aligned}
l(\varphi)&=\Int_{\Omega}\left[(i\beta)^{2}\Int_{0}^{\infty}g(s)\omega_{3}(s)ds\varphi dx+(1-g_{0)}\int_{0}^{\infty}g(s)\nabla\omega_{3}(s)ds\nabla\varphi dx+(f_{2}+i\beta f_{1})\varphi dx\right] \\&+\Int_{\Omega}\Int_{0}^{\infty}g(s)\omega_{4}(s)ds \varphi dx+i\beta\int_{\Gamma_{1}}u_{0}(\sigma)\varphi(\sigma)d\sigma
\end{aligned}
\end{equation}
It is clear that $b$ is a bilinear and continuous form on $H_{\Gamma_{0}}^{1}(\Omega)$ and $l$ is linear and continuous form on $H_{\Gamma_{0}}^{1}(\Omega)$. On the other hand (\ref{mil2}) implies that there exists a positive constant $a_{0}$ such that

\begin{equation}
\begin{aligned}
a(\bar{u},\bar{u})&=\Int_{\Omega}(i\beta)^{2}|\bar{u}|^{2}dx+(1-g_{0)})\int_{\Omega}|\nabla\bar{u}|^{2}dx+i\beta\int_{\Gamma_{1}}|\bar{u}(\sigma)|^{2}d\sigma\\&
\geq a_{0}\|\bar{u}\|_{2}^{2} \  \  \forall \bar{u}\in H_{\Gamma_{0}}^{1}(\Omega),
\end{aligned}
\end{equation}
which implies that $b$ is coercive. Therefore, using the Lax-Milgram Theorem, we conclude that (\ref{tas6}) has a unique solution $\bar{u}$ in $H_{\Gamma_{0}}^{1}(\Omega)$. By classical regularity arguments, we conclude that the solution $\bar{u}$ of (\ref{tas2}) belongs into $H^{2}(\Omega)\cap H_{\Gamma_{0}}^{1}(\Omega)$. Consequently, using  (\ref{tas2}) and (\ref{tass2}), we deduce  that (\ref{ak3}) has a unique solution $V\in D(\mathcal{A})$. This proves that $(i\beta -\mathcal{A})$ is surjective.
\begin{lemma}
The resolvent operator of $\mathcal{A}$ satisfies (\ref{ak2}).
\end{lemma}
Suppose that condition (\ref{ak2}) is false. By Banach-steinhaus Theorem (\cite{DS}), there exists a sequence of real numbers $\beta_{n} \to +\infty$ and a sequence of vectors

\begin{equation}\label{est1}
Z_{n}=(u_{n},\upsilon_{n}, \omega_{n},\eta_{n})^\mathsf{T}\in D(\mathcal{A}) \quad  \hbox{with} \ \|Z_{n}\|_{\mathcal{H}}=1
\end{equation}
such that
\begin{equation}\label{est2}
  \quad \|(i\beta_{n}I-\mathcal{A})Z_{n}\|_{\mathcal{H}} \to 0 \ \hbox{as} \  n \to \infty.
\end{equation}
That's
\begin{eqnarray}
\!\!\!\!\!\!\!\!\!\!\!\!&&\!\!\!\!\!\!\!\!\!\!\!\!\left(i\beta_{n}u_{n}-\upsilon_{n}\right)\equiv f_{n}\to 0,\hbox{ in }   H_{\Gamma_{0}}^{1}(\Omega)\label{seq11}\\
\!\!\!\!\!\!\!\!\!\!\!\!\!\!\!\!\!\!\!\!&&\!\!\!\!\!\!\!\!\!\!\!\!\!\!\!\!\!\!\!\!\left(i\beta_{n}\upsilon_{n}-(1-g_{0})\Delta u_{n}-\int_{0}^{\infty}g(s)\Delta \eta_{n}(s)ds\right)\equiv g_{n} \to 0,\hbox { in } L^{2}(\Omega)\label{seq12}\\
\!\!\!\!\!\!\!\!\!\!\!\!&&\!\!\!\!\!\!\!\!\!\!\!\!\left(i\beta_{n}\omega_{n}+\Frac{\partial u_{n}}{\partial\nu}+\int_{0}^{\infty}g(s)\Frac{\partial \omega_{n}(s)}{\partial \nu}ds\right)\equiv h_{n} \to 0,\hbox{ in }  L^{2}(\Gamma_{1})\label{seq13} \\
\!\!\!\!\!\!\!\!\!\!\!\!&&\!\!\!\!\!\!\!\!\!\!\!\!\left(i\beta_{n}\eta_{n}+\Frac{\partial \eta_{n}}{\partial s}-\upsilon_{n}\right)\equiv k_{n} \to 0,\hbox{ in } \mathcal{M}\label{seq14}.
\end{eqnarray}
Our aim is to derive from (\ref{est2}) that $\|Z_{n}\|_{\mathcal{H}}$ converges to zero, thus there is a contradiction.\\

\begin{equation}
\left|\Re\left\langle(i\beta_{n}I-\mathcal{A})Z_{n},Z_{n}\right\rangle_{\mathcal{H}}\right|\leq\|(i\beta_{n}I-\mathcal{A})Z_{n}\|_{\mathcal{H}}
\end{equation}
Using the hypotheses on $g$ , we find that
\begin{equation}\label{hit22}
\eta_{n} \to 0 \  \  \hbox{in } \ \  L_{g}^{2}(\mathbb{R}_{+}; H_{\Gamma_{0}}^{1}(\Omega))
\end{equation}
and
\begin{equation}\label{hit23}
\eta_{n}(s)=\left(\Int_{0}^{s}e^{i\beta y}k_{n}(y)\right)e^{-i\beta s}+\left(\Int_{0}^{s}e^{i\beta y}\upsilon_{n}(y)dy\right)e^{-i\beta s}.
\end{equation}
By exploiting the convergence (\ref{hit22}) and (\ref{hit23}), we can deduce from (\ref{seq11}) that

\begin{equation}
\upsilon_{n} \to 0 \  \  \hbox{in } \ \  L^{2}(\Omega) \  \hbox{and} \  u_{n} \to 0 \  \  \hbox{in } \ \  L^{2}(\Omega).
\end{equation}
Now, multiplying the equation (\ref{seq11}) by $\upsilon_{n}$  and (\ref{seq12}) by $u_{n}$, adding them and taking the real parts , we obtain

\begin{equation}\label{seq15}
-\|\upsilon_{n}\|_{2}^{2}+(1-g_{0})\|\nabla u_{n}\|_{2}^{2}+\int_{0}^{\infty}g(s)\nabla\eta_{n}(s)\nabla u_{n}(t)ds \to 0 \  \hbox{in} \ L^{2}(\Omega).
\end{equation}
Due to Young's inequality, we have for any $\delta>0$

\begin{equation}
\begin{aligned}
\int_{0}^{\infty}g(s)\int_{\Omega}\nabla
\eta_{n}(s)\nabla
u_{n}(t)dxds&\leq\int_{0}^{\infty}g(s)\left(\frac{1}{4\delta}\|\nabla\eta_{n}(s)\|_{2}^{2}+\delta\|\nabla
u_{n}\|_{2}^{2}\right)ds\\&
\leq\delta\left(\int_{0}^{\infty}g(s)ds\right)\|\nabla
u_{n}\|_{2}^{2}+\frac{1}{4\delta}\int_{0}^{\infty}g(s)\|\nabla\eta_{n}(s)\|_{2}^{2}ds
\\&
=\delta g_{0}\|\nabla
u_{n}\|_{2}^{2}+\frac{1}{4\delta}\|\eta\|_{L_{g}^{2}}^{2}.
\end{aligned}
\end{equation}
Replacing the last inequality in (\ref{seq15}), for $\delta$ sufficiently small  we  get

\begin{equation}
\nabla u_{n} \to 0 \  \  \hbox{in } \ \  L^{2}(\Omega).
\end{equation}
Consequently we have

\begin{equation}
u_{n} \to 0 \  \  \hbox{in } \ \  H_{\Gamma_{0}}^{1}(\Omega).
\end{equation}
By using (\ref{seq13}) and trace theorem we get

\begin{equation}
\omega_{n} \to 0 \  \  \hbox{in } \ \  L^{2}(\Gamma_{1})
\end{equation}
which contradicts (\ref{est1}) . Thus (\ref{ak2}) is proved.

\end{proof}
\end{proof}

 \end{document}